\documentclass[12pt,reqno]{amsart}
\usepackage[latin1]{inputenc}
\usepackage{amsmath,amssymb,amsfonts,amsthm}
\usepackage{amscd,amstext,amsxtra,amsopn,array,url,verbatim,mathrsfs}
\usepackage{graphicx}
\usepackage{amscd,url,amstext,amsxtra,amsopn}
\usepackage{verbatim}
\usepackage{fullpage}
\usepackage{times}
\usepackage{amssymb}
\usepackage{subfigure}
\usepackage{url}
\usepackage{graphicx}
\usepackage{hyperref}
\usepackage{epsf}
\usepackage[latin1]{inputenc}
\usepackage[T1]{fontenc}
\usepackage{lmodern}

\newtheorem{corollary}{Corollary}[section]
\newtheorem{definition}{Definition}[section]

\newtheorem{lemma}{Lemma}[section]

\newtheorem{theorem}{Theorem}[section]
\newtheorem*{thm}{Theorem}
\newtheorem*{conjecture}{Conjecture}

\theoremstyle{remark}
\newtheorem{remark}{Remark}[section]
\numberwithin{equation}{section}
\numberwithin{table}{section}
\numberwithin{figure}{section}
\hypersetup{colorlinks=true,urlcolor=blue,linkcolor=blue,citecolor=blue}

\begin{document}

\pagenumbering{arabic}

\title{Distribution of squares modulo a composite number}
\author{Farzad Aryan}
\maketitle

\begin{abstract}
In this paper we study the distribution of squares modulo a square-free number $q$.
We also look at inverse questions for the large sieve in the distribution aspect and
we make improvements on existing results on
the distribution of $s$-tuples of reduced residues.
\end{abstract}

\section{Introduction}
 
In this paper we are mainly concerned with the distribution of subsets of integers that are not additively structured, though we
will also prove results for sets that are additively structured. We begin by studying squares, which is the model example of a non-additively structured set. We continue with more complicated non-additively structured sets. The final part will be the study of the higher central moments of  $s$-tuples of reduced residues.

\subsection*{The distribution of squares modulo $q$}

For $q$ square-free, we call an integer $s$ a square modulo $q$
when $s$ is a square modulo $p$ for all primes $p$ dividing $q$. Note that we count $0$ as a square. Several authors have studied the
distribution of spacings between squares modulo $q$.
For $q$ prime, a theorem of Davenport~\cite{DA} shows that the probability of two consecutive squares modulo $q$ being spaced $h$ units
 apart is asymptotically $2^{-h}$ as $q$ tends to infinity.
For $q$ square-free, Kurlberg and Rudnick~\cite{RK} have shown that
the distribution of spacings between squares approaches
a Poisson distribution  as $\omega(q)$ tends to infinity,
where $\omega(q)$ is the number of distinct prime divisors of $q$.
\begin{thm}[Kurlberg and Rudnick]
Let $\square$ be the symbol that denotes the word square and 
\begin{align*}
    \overline{s} = \frac{q}{\# \{x :x \text{ is $\square$ modulo }q \}},
\end{align*}
and let $I$ be an interval in $\mathbb{R}$ that does not contain zero. Then
\begin{equation}
\label{eq1.4}
    \frac{\#\{(x_1, x_2) : x_1-x_2 \in  \overline{s}I : x_1, x_2 \text{ are $\square$ mod } q\}}
    {\# \{x: x \text{ is a $\square$ modulo }q \}}
    = |I| + O\big( \frac{1}{s^{1-\epsilon}} \big).
\end{equation}
\end{thm}
Note that $ \overline{s}$ is the mean spacing in the set of squares modulo $q$ and the ``probability'' of a random integer being a square modulo $q$ is $1/ \overline{s} $,
which approximately is $ 1/2^{\omega(q)}.$
The results we prove in this paper are, more or less, in the spirit of papers written by Montgomery and Vaughan~\cite{mv} and Hooley~\cite{H1, H2, H3}. These articles answer Erd\H{o}s' question in~\cite{er} regarding the gaps between consecutive reduced residues. The reduced residues modulo $q$ are the integers $a_i$, $1 = a_1 < a_2 < \ldots < a_{\phi(q)} < q$, that are relatively prime to $q$.  Erd\H{o}s~\cite{er} proposed the following conjecture for the second moment of the gap between consecutive reduced residues: For $\lambda=2$ we have
\begin{equation}
\label{er-conj}
V_\lambda (q) = \sum_{i=1}^{\phi(q)} (a_{i+1}- a_i )^ \lambda \ll qP^{1-\lambda},
\end{equation}
where $P=\phi(q)/q$ is the ``probability'' that a randomly chosen integer is relatively prime to $q$. Hooley~\cite{H1} showed that \eqref{er-conj} holds for all $0<\lambda<2.$  For $\lambda=2$,  Hausman and Shapiro \cite{sh} gave a weaker bound than \eqref{er-conj}. Finally, Montgomery and Vaughan~\cite{mv} succeeded in proving the conjecture, showing that \eqref{er-conj} holds for all $\lambda >0$. The key ingredient in the proof of the results of \cite{H1} and \cite{sh} is the variance of the random variable
\begin{align*}
    \mathcal{R}_{h}(n)= \#\{m \in [n,n + h] : m \text{ is a reduced residue modulo } q\}.
\end{align*}
In \cite{mv} Variance and also higher central moments of $\mathcal{R}_{h}$ were studied.
Motivated by the above results, we consider the variance of the following random variable. Let $n$ be an integer chosen uniformly at random in $\{1, 2, \ldots , q\}$,
and define $\mathcal{X}_h$ by
\begin{align*}
    \mathcal{X}_{h}(n)= \#\{s \in [n,n + h] : s \text{ is a $\square$ modulo } q\}.
\end{align*}

\begin{theorem}
\label{th0-2} Let $q$ be a square-free number and $P=\phi(q)/q$. Then as an upper bound we have
\begin{equation}
\label{eq1.5}
    \frac{1}{q} \displaystyle{\sum_{n=0}^{q-1}\bigg( \sum_{\substack{m=1 \\ n+m \hspace{1 mm} \text{is } \square \hspace{1 mm}
     \text{mod}  \hspace{1 mm}q}}^{h} 1 - \frac{h}{2^{\omega(q)}P} \bigg)^{2}}
    \leq \frac{h}{2^{\omega(q)}P} \prod_{p|q} \Big( 1+\frac{1}{\sqrt{p}} \Big),
\end{equation}
and as a lower bound we have
\begin{equation}
    \frac{h}{4^{\omega(q)}P}\sum_{\substack{r>h^2 \\ r|q}}
     \prod_{p|r} \Big( 1 - \frac{3}{\sqrt{p}} \Big)
    \ll \frac{1}{q} \displaystyle{\sum_{n=0}^{q-1}\bigg( \sum_{\substack{m=1 \\ n+m \hspace{1 mm} \text{is } \square \hspace{1 mm}  \text{mod}
    \hspace{1 mm}q}}^{h} 1 - \frac{h}{2^{\omega(q)}P} \bigg)^{2}}.
\end{equation}
Moreover, if the prime divisors of $q$ are all congruent to $3$ modulo $4$ then we have the sharper bound
\begin{equation}
\label{3mod4}
 \frac{1}{q} \displaystyle{\sum_{n=0}^{q-1}\bigg( \sum_{\substack{m=1 \\ n+m \hspace{1 mm} \text{is } \square \hspace{1 mm}
 \text{mod}  \hspace{1 mm}q}}^{h} 1 - \frac{h}{2^{\omega(q)}P} \bigg)^{2}} \leq  \frac{h}{2^{\omega(q)}P^2}.
\end{equation}

\end{theorem}
\noindent Note that in \eqref{eq1.5},
$\prod_{p|q} ( 1 + p^{-1/2} ) \ll 2^{\sqrt{\omega(q)}}$
which is much smaller than $2^{\omega(q)}$
for large $\omega(q)$.
For the mean of $\mathcal{X}_h,$ we have ${\rm \bf E}(\mathcal{X}_h)= \frac{h}{2^{\omega(q)}P}$ and therefore the left hand side of~\eqref{eq1.5} is equal to the variance of $\mathcal{X}_{ h}$, which we denote by ${\rm \bf Var}(\mathcal{ X}_{ h})$. Consequently, Theorem \ref{th0-2} implies the following upper bound:
\begin{align*}
    {\rm \bf Var}(\mathcal{ X}_{ h}) \leq
    \prod_{p|q} \Big( 1 + \frac{1}{\sqrt{p}} \Big)
    \cdot {\rm \bf E}(\mathcal{ X}_{ h}),
\end{align*}
whereas the trivial upper bound is
\begin{align*}
    {\rm \bf Var}(\mathcal{ X}_{ h}) \leq {\rm \bf E} (\mathcal{ X}_{h})^2.
\end{align*}

\begin{remark}
Let $N$ be the set of quadratic non-residues modulo a prime $p$. The reason for the better bound \eqref{3mod4} is that, when $p \equiv 3$ mod $4$, the size of the Fourier coefficient $\sum_{n \in N}e(n/p)$ of $N$, is smaller than when $p \equiv 1$ mod $4.$
\end{remark}
In 1936 Cramer \cite{C}, assuming the Riemann hypothesis (RH),
showed the following result concerning the average gap between consecutive primes:
\begin{equation}
    \label{eq-1}
    \sum_{p_n <x}(p_{n+1}-p_n)^2 \ll x(\log x)^{3+\epsilon}.
\end{equation} This bound was the inspiration of Erd\H{o}s'  conjecture~\eqref{er-conj}. Using Theorem \ref{th0-2}, we prove an analogous result for gaps between squares.
\begin{corollary}
\label{corr-01}
Let $s_i$ be the squares modulo $q$ in increasing order.
Then
\begin{equation}
\label{cor--1}
    \frac{1}{q} \sum_{s_i < q} (s_{i+1} - s_{i})^2
    \ll 2^{\omega(q)} P (\log q) \prod_{p|q} \Big( 1 + \frac{1}{\sqrt{p}} \Big).
\end{equation}
\end{corollary}
\begin{remark}
It seems plausible that the factor $\prod_{p|q} (1 + p^{-1/2})$ can be removed from the right hand side of \eqref{cor--1}. Also, it seems difficult to estimate the higher moments in Corollary 0.1. Indeed, in the simple case where $q$ equals a prime number $p$, a good estimation of the higher moments would imply that the gap between two consecutive quadratic residues is less than $p^{o(1)}.$ Note that the best known bound obtained by Burgess \cite{bur} is $p^{1/4+ o(1)}.$
\end{remark}
\noindent An important property of the squares that we use in the proof of Theorem \ref{th0-2} is the following:
 For $a\neq 0$ modulo $p$ we have
\begin{equation}
\label{eq1.5.1}
 \bigg|\sum_{\substack{{s \text{ is $\square$ mod } p}}}e\Big(\frac{sa}{p}\Big)\bigg| \ll \sqrt{p}.
\end{equation} In the language of Fourier Analysis, this property means that all of the non-trivial Fourier coefficients of the set of squares have square root cancellation. In the context of this paper we denote the property of having small Fourier coefficient as being ``non-additively structured''. In the next section we generalize Theorem \ref{th0-2} for all the sets that are not additively structured. We also use similar ideas to study a problem related to additive combinatorics which is known as the inverse conjecture for the large sieve.

\subsection*{Relation with the inverse conjecture for the large sieve}

In this section we consider the inverse conjecture for the large sieve. We also introduce the notions of  ``additively structured'' and ``non-additively structured'' sets and study the distribution of these sets. Based on these ideas we formulate a refined version of the inverse conjecture. Roughly speaking, we say that a subset of $\mathbb{Z}/p\mathbb{Z}$  is not additively structured if all of its non-trivial Fourier coefficients have square root cancellation. On the other hand, being additively structured means there exist at least one large Fourier coefficient. Having a large Fourier coefficient is equivalent to saying that the set has many quadruples $(x_1, x_2, x_3, x_4)$ such that $x_1+ x_2=x_3+ x_4,$ which explains the reason for choosing  the ``additive structure'' terminology. \\
\\
Let $A$ be a finite set of integers with the property that the reduced set $A$ (mod $p$) occupies at most $(p + 1)/2$ residue classes modulo $p$ for every prime $p|q$.
In other words, for $p|q$ and $\Omega_p\subseteq \mathbb{Z}/p\mathbb{Z} $ with $|\Omega_p|=(p-1)/2$,
$A$ is obtained by sieving $[1, X]$ by all the congruence classes in $\Omega_p$.
The inverse problem for the large sieve is concerned with the size of $A$ (see \cite{H.V}). In the case where $q$ is equal to the product of all primes less than $\sqrt{X}$, using the large sieve inequality one can show that $|A| \ll \sqrt{X}$. The following is the formulation of the conjecture by Green \cite{Gr}.

\begin{conjecture}[\textbf{Inverse conjecture for the large sieve}]
For every prime number $p<\sqrt X,$ let $\Omega_p \subseteq \mathbb{Z}/p\mathbb{Z}$ with $|\Omega_p|=(p-1)/2.$
Let $A \subseteq \{1, 2, \ldots, X\}$ be the set obtained by sieving out the residue classes in $\Omega_p$ for $p < X$.
Then $|A| \ll X^{\epsilon}$ unless $A$ is contained in the set of values of a quadratic polynomial $f(n) =an^2 + bn + c$,
with the possible exception of a set of size $X^{\epsilon}$.
\end{conjecture}
\begin{remark}
 This has been stated informally in the literature as follows. If the size of $A$ is not too small then $A$ should possesses an ``algebraic'' structure. The problem with
 this statement is that a formal definition for possessing an ``algebraic'' structure has not been given. Although it seems that any set with ``algebraic'' structure is not additively structured, the reverse may not be true.
\end{remark}

Here our aim is to look at this problem from the distributional aspect. We consider $A$ to be a subset of an interval larger than the interval $[1, X]$.
We fix $A$ to be a subset of $\{1, 2, \ldots, \hspace{1 mm} q\}$.
This set shall be defined by sieving out congruence classes in $\Omega_p$ for all $p|q$, $|\Omega_p|=(p-1)/2$.
Next we let $n$ be an integer picked uniformly at random in $\{1, 2, \ldots, \hspace{1 mm} q\}$,
and define the random variable $\mathcal{ Y}_h$  by
\begin{equation}
    \label{eq-r0.5}
    \mathcal{Y}_h(n) = |[n,n+h] \cap A |.
\end{equation}
Since $|\Omega_p|=(p-1)/2$, the Chinese Remainder Theorem implies that
\begin{align*}
    |A| = \prod_{p|q} \Big(\frac{p+1}{2}\Big).
\end{align*}
\noindent \textbf{Question}: How is $A$ distributed modulo $q$? \\
\\
\noindent We will prove a result which shows that if for all $p|q,$ $\Omega_p$ is not additively structured, then $A$ is well distributed. In the other direction we show some partial results in the case that $\Omega_p$ is additively structured. The latter result indicates that $A$ is far from being well distributed. To make the notion of being well distributed more clear in the context of this paper, let $A \subseteq [1, q]$ and define $
\text{Prob}(x \in A)= |A|/q.$
We say that $A$ is well distributed if any interval of length $h$ inside $[1, q],$ contains $h\frac{|A|}{q}(1+ o(1))$  elements of $A.$  \\
\\
Now we introduce the notion of a set that is ``not additively structured''.  We describe this using the example of squares. In this case $\Omega_p$ is the set of non-quadratic residues modulo $p$. In other words in order to end up with squares after sieving, we need to sieve out integers congruent to non-quadratic residues modulo each prime $p|q$.  Inspired by the property of squares mentioned in the Equation \eqref{eq1.5.1}, we have the following definition.

\begin{definition}[\textbf{Not additively structured}]
For $p$ a prime number we say that $\Omega_p\subseteq \mathbb{Z}/p\mathbb{Z} $ is not additively structured if for all $a\neq 0$ modulo $p$,
\begin{equation}
    \label{eq1.6}
    \Big|\sum_{x \in \Omega_p } e\Big( \frac{ax}{p} \Big) \Big| < c_p\sqrt{p},
\end{equation}
where $c_p$ depends on $p$ and satisfies $c_p \ll \log p$.
\end{definition}

\noindent We will give two examples of sets that are not additively structured.\\
\\
\textbf{Example 1.} By using the following theorem of Weil we can show that the image of a polynomial $P$ is not additively structured under the following condition: For every $y \in \text{Im}(P)$ the equation $P(x)=y$ $(\text{mod }p)$  has the same number of solutions, with the exception of a subset $\mathcal{E}$ of the image with $|\mathcal{E}| \ll \sqrt{p}.$
\begin{thm}[Weil]
Let $P \in \mathbb{Z}[X]$ be a polynomial of degree $d > 1$. Let $p$ be a prime such that $\gcd(d, p)=1$. Then we have
\begin{equation}
\label{weil}
\displaystyle{\Big|\sum_{x  (\text{mod }p) } e\Big(\frac{P(x)}{p}\Big)\Big|< (d-1)\sqrt{p}}.
\end{equation}
\end{thm}

\noindent \textbf{Example 2.} Another example of a set that is not additively structured is
\begin{equation}
\label{Omega_{K,p}}
\Omega_{K,p} := \{x+ y: 1 \leq x, y\leq p-1  \text{ and } xy \equiv 1 \text{ mod }p \}.
\end{equation}
This can be shown by using the Weil bound on Kloosterman sums. The size of $\Omega_{K, p} $ is $(p+1)/2$. One open problem regarding this asks about existence of small residue classes with small reciprocal. More precisely let
\begin{equation}
\label{M_p1}
 M_p:= \min_{x \neq 0 \text{ mod } p} \big \{ \max \{x, x^{-1}\} \big\}.
\end{equation}
Then the question is how small can $M_p$ be? As an application of the Weil bound on Kloosterman sums  one can show that
 $M_p \leq 2 (\log p) p^{3/4}$ (see \cite{He}). It seems natural to conjecture that $M_p \leq p^{1/2 + \epsilon}.$ In fact Tao \cite{Tao} even suggested that $M_p=O(p^{1/2})$ might be possible. Using Theorem \ref{th-0.2} one can show that there exist $x$ modulo $p$ such that
\begin{equation}
\label{M_p2}
 x+x^{-1} \text{ mod } p \leq p^{1/2+\epsilon}.
\end{equation}
 For such $x$ if
$x+_{\mathbb Z} x^{-1} <p,$  then \eqref{M_p2} would imply the conjectural bound for $M_p$.
However if $x+_{\mathbb Z} x^{-1} \geq p,$ then \eqref{M_p2} does not give any useful information. Thus it would be interesting to look at the
distribution of the set
\begin{equation}
\label{M_p3}
 \Omega^{\prime}_{K, p}:= \big \{ x+ x^{-1} : 1\leq x, x^{-1} \leq p-1 \text{ and } x+_{\mathbb Z} x^{-1} <p\big \}.
\end{equation}
If $\Omega^{\prime}_{K, p}$ were not additively structured then it would imply the conjectural bound for $M_p$.
However in Theorem \ref{Thm-M_p}, we will show that this is not the case and $\Omega^{\prime}_{K, p}$ is not well distributed modulo $p$. Consequently one way to attack the
 conjectural bound on $M_p$ would be to find a proper subset of $\Omega^{\prime}_{K, p}$ which is not additively structured. Another way would be
to add certain elements to $\Omega^{\prime}_{K, p}$ in order to make a set that is not additively structured.
  \\
\\
\noindent  We show that if $\Omega_p$ is not additively structured then $A$  is well distributed.
\begin{theorem}
\label{th-0.2}
Let $\mathcal{ Y}_h$ be as \eqref{eq-r0.5}. Then if $\Omega_p$ is not additively structured i.e., satisfies \eqref{eq1.6} and $|\Omega_p|=(p-1)/2$, then we have
\begin{equation}
\label{eq-c_p}
\frac{1}{q} \sum_{n=0}^{q-1}\Bigg( \sum_{\substack{m \in [n, n+h] \\ m \notin \Omega_p \text{ mod } p \\ \forall p|q}}1- h\prod_{p|q} \bigg( \frac{p+1}{2p}\bigg) \Bigg)^{2}  \ll h\prod_{p|q} \bigg(\Big(\frac{p+1}{2p}\Big)^2 +c^{ 2}_{p} \bigg),
\end{equation}
or equivalently
\begin{equation}
\label{eq1.61}
{\rm \bf Var}(\mathcal{ Y}_h) \ll {\rm \bf E}(\mathcal{ Y}_h)\prod_{p|q}\bigg(\frac{p+1}{2p}+ \frac{2c^{ 2}_{p}p}{p+1}\bigg),
\end{equation}
where $c_p$ is the constant in \eqref{eq1.6}.
\end{theorem}
\begin{remark}
 Note that $c_p$ can never be too small. In fact one can get a lower bound $c_p> 1/2.$ As a result the right hand side of \eqref{eq-c_p} is always bigger than $h/2^{\omega(q)}.$
\end{remark}

\begin{remark}
Note that the trivial upper bound on ${\rm \bf Var}(\mathcal{ Y}_h)$ is ${\rm \bf E}(\mathcal{ Y}_h)^2.$ In section 3 we prove a  more general result without the restriction $|\Omega_p|=(p-1)/2$ (see Lemma \ref{t4.1}).
\end{remark}
\begin{remark}
 By taking $\Omega_p$ equal to the set of quadratic non-residues in Theorem \ref{th-0.2}, we obtain Theorem \ref{th0-2}.
\end{remark}

\noindent Returning to the inverse conjecture for the large sieve, Green and Harper \cite{G.H} proved the conjecture when $\Omega_p$
is an interval and gave a non-trivial result when $\Omega_p$ has certain additive structure. This brings us to the definition of a set with additive structure.
\begin{definition}[\textbf{Additively structured}]
For $p$ a prime number we say that $\Omega_p\subseteq \mathbb{Z}/p\mathbb{Z} $ is additively structured if there exist $a\neq 0$ modulo $p$,
\begin{equation}
    \label{eq1.6.1}
    \Big|\sum_{x \in \Omega_p } e\Big( \frac{ax}{p} \Big) \Big| \geq C_{p} p,
\end{equation}
where $C_p$ depends on $p$ and here we consider $C_p \gg \log^{-1} p.$
\end{definition}
\begin{remark}
 Note that additively structured is the extreme opposite of not additively structured, since the opposite of not additively structured means every set has a Fourier coefficient just bigger than $p^{1/2 + \epsilon},$ while being additively structured means there exists a Fourier coefficient bigger than $p^{1 - \epsilon}.$
\end{remark}

Let $\displaystyle{\Omega_p=\{0, 2, 4, \ldots,  p-1\}},$  for all $p|q$. Note that this set is additively structured since
for $a=(p+1)/2$ we have
 $$\Bigg| \sum_{x\in \Omega_p} e\Big(\frac{xa}{p}\Big) \Bigg|= \Bigg|\frac{e\big(\frac{1}{2p}\big)+1}{e\big(\frac{1}{p}\big)-1}\Bigg| \geq \frac{p}{\pi}.$$

\noindent For the set $A$ we prove a result which shows that $A$ is far from being well distributed.
\begin{theorem}
\label{coro1}
Let $\Omega_p=\{0, 2, 4, \ldots , p-1\}$ and $\mathcal{ Y}_h$ be as \eqref{eq-r0.5}. Assume $\displaystyle{q=p_1, \ldots, p_{\lfloor \log X\rfloor}}  $, where $X<p_i <2X$ and $|p_2-p_1|\ll \log p_1$. Then for every integer $h < \frac{X^2}{\log X }$ we have that
$$
\frac{1}{q}\sum_{n=0}^{q-1}\Bigg( \sum_{\substack{m \in [n, n+h] \\ m \notin \Omega_p \text{ mod } p\\ \forall p|q}}1- h\prod_{p|q} \bigg( \frac{p+1}{2p}\bigg) \Bigg)^{2}  \gg \Big(\frac{h}{2^{\omega(q)}P}\Big)^2, $$
or equivalently
$${\rm \bf Var}(\mathcal{ Y}_h) \gg {\rm \bf E}(\mathcal{ Y}_h)^2.  $$
\\
\end{theorem}

Theorem \ref{th-0.2} shows a connection between non-additive structure in sets $\Omega_p$ and well distribution of $A$.
\noindent Theorems \ref{coro1} shows a connection between the additive structure of the sets $\Omega_p$ and $A$ not being well distributed.
Recall that $A$ is obtained by sieving out the congruence classes in $\Omega_p.$  In the inverse conjecture for the large sieve, there is a similar connection between the size of the sifted set and the additive structure of $\Omega_p$. More precisely, if the size of the sifted set $A$ is not too small, then $A$ is the image of a quadratic polynomial and from Example 1 we know that the image of a quadratic polynomial is not additively structured. Thus if the size of $A$ is large then $A$ is not additively structured.  Inspired by this observation it seems natural to refine the inverse conjecture for the large sieve in terms of the additive structure of $A.$ Now we state our conjecture. \\
\begin{conjecture}
Let $A$ be the subset of $[1, X]$  obtained by sieving out congruence classes in $\Omega_p$ for $p < X^{1/2}$.
Moreover assume that for each $p,$ $\Omega_p$ is additively structured i.e. $\Omega_p$ has the property that there exist $a \neq 0$ modulo $p$ such that
\begin{align*}
    \bigg| \sum_{x \in \Omega_p} e\big(\frac{ax}{p} \big) \bigg|\gg C_p p,
\end{align*}
with $C_p \gg \log^{-1} p.$ Then
\begin{align*}
    |A| \ll X^{\epsilon}.
\end{align*}
\end{conjecture}
\noindent Harper and Green \cite{G.H} proved a non-trivial bound for the size of $A$ in the above conjecture. They  proved that if  $\Omega_p$ has many quadruples $(x_1, x_2, x_3, x_4)$ such that $x_1+ x_2= x_3+ x_4,$ then there exists $c>0$ such that $|A| \ll X^{1/2 - c}.$  Note that the quadruple condition is equivalent to $\Omega_p$ having a  large Fourier coefficient. (Larger than $p^{1-\epsilon}.$)
\\
\\
\noindent To finish this part of the article we state a result regarding the distribution of $\Omega^{\prime}_{K, p}$ from Example 2. Note that
\begin{equation}
\label{th4p0}
|\Omega^{\prime}_{K, p}|=
\begin{cases}
\frac{p+1}{4} &\text{if $p \equiv 3 \text{ mod 4}$},\\
\frac{p-1}{4} &\text{if $p \equiv 1 \text{ mod 4}$.}
\end{cases}
\end{equation}
\begin{theorem}
\label{Thm-M_p}
 Let $\Omega^{\prime}_{K, p}$ be as in \eqref{M_p3}. Then for $h<p/2$ we have
\begin{equation}
 \label{th4p1}
\frac{1}{p}\sum_{n=0}^{p-1}\Bigg( \sum_{\substack{m \in [n, n+h] \\ m \in \Omega^{\prime}_{K, p} }}1- \frac{h}{p}|\Omega^{\prime}_{K, p}|
  \Bigg)^{2}  \gg h^2.
\end{equation}
\end{theorem}
\noindent In the last part of this article we study the distribution of $s$-tuples of reduced residues.
 Although the following theorem is independent than previous results, the techniques are very similar. In particular Lemma \ref{Lemma5} will be applied in all theorems.\\

\subsection*{Higher central moments for the distribution of $s$-tuples of reduced residues}

Let $$\mathcal{D}=\lbrace h_1, h_2 , \ldots, h_s \rbrace,$$ and $\nu_p(\mathcal{D})$ be the number of distinct
elements in $\mathcal{D}$ mod $p$. We call $\mathcal{D}$  \emph{admissible} if $\nu_p(\mathcal{D})<p$ for all primes $p$.
We call $(a+h_1,\ldots, a+h_s)$ an \emph{$s$-tuple of reduced residues} if each element $a+h_i$ is coprime to $q$.
In our previous results we were only able to calculate the variance and could not obtain any estimate for higher moments.
The reason for this, in a sieve-theoretic language, is that when $|\Omega_p|=(p-1)/2,$ as $q$ tends to infinity the dimension of the sieve also tends
to infinity. However, if we fix our admissible set and look at the distribution of $s$-tuples of reduced residues, then the dimension stays bounded and
consequently we are able to derive results for higher moments. Let $k_q(m)$ be the characteristic function of reduced residues, that is to say
\begin{equation*}
    k_q(m)=
    \begin{cases}
    1 &\text{if $\gcd(m,q)=1$},\\
    0 &\text{otherwise.}
    \end{cases}
\end{equation*}
 The generalization of Erd\H{o}s' conjecture, i.e.
\begin{equation}
\label{er-conj-gen}
V^{\mathcal{D}}_2 (q) = \sum_{\substack{(a_i+h_j, q)=1 \\ h_j \in \mathcal{D}}} (a_{i+1}- a_i )^ 2 \ll qP^{-s},
\end{equation}
 concerns the gap between $s$-tuples of reduced residues. In order to prove the generalization of Erd\H{o}s' conjecture (see~\cite{er} and~\cite{F.A}),
the author in~\cite{F.A} studied the $k$-th moment of the distribution of the $s$-tuples of reduced residues:
Let
\begin{align}
\notag  M^{\mathcal{D}}_k(q, h):=  \sum_{n=0}^{q-1}\left( \sum_{m=1}^{h}k_q(n+m+h_1)\ldots k_q(n+m+h_s)- h\prod_{p|q} \bigg(1-\frac{\nu_p(D)}{p}\bigg) \right)^{k}.
\end{align}
In the case  $s=1,$ i.e. ${\mathcal{D}}=\{0\},$  and $k<2$ this was studied by Hooley ~\cite{H1} who found an upper bound for $M^{\{0\}}_2(q, h).$ Hausman and Shapiro ~\cite{sh} gave an exact formula for $M^{\{0\}}_2(q, h).$ Their formula immediately gives the upper bound $M^{\{0\}}_2(q, h) \leq qhP.$ Finally, for a fixed natural number $k$, Montgomery and Vaughan~\cite{mv} showed
\begin{equation}
\label{mv-main}
 M^{\{0\}}_k(q, h) \leq q(hP)^{k/2}+ qhP.
\end{equation}
 For a fixed admissible set ${\mathcal{D}}$ it was proven in~\cite{F.A} that
\begin{align}
\label{Me-notmain}
    M^{\mathcal{D}}_k(q, h) \ll_{s, k}  qh^{k/2} P^{-2^{ks}+ks}.
\end{align}
This was enough to get the generalization of Erd\H{o}s' conjecture, however the method failed to get bounds as strong as \eqref{mv-main}. In the last section of this paper we improve \eqref{Me-notmain}.
\begin{theorem}
\label{thm4}
Let $\displaystyle P = \phi(q)/q$.
For $\displaystyle h<\exp\big( {\frac{1}{k P^{1/s}}}\big) $, we have
\begin{equation}
    \label{eq7}
    M^{\mathcal{D}}_k(q, h) \ll_{s, k} q(hP^s)^{k/2}
\end{equation}
and in general
\begin{equation}
\label{eq8}
M^{\mathcal{D}}_k(q, h) \ll_{s, k} qh^{k/2}P^{sk-s^2\frac{k}{2}}.
\end{equation}
\end{theorem}
\begin{remark}
Note that \eqref{eq7} is the best possible upper bound and it matches the upper bound derived from probabilistic estimates (see \cite[Lemma 2.1]{F.A}).
\end{remark}
\noindent The open question that remains here is  whether or not the bound \eqref{eq8} is sharp. In other words, is there an admissible set $\mathcal{D}$ such that for $h \geq \exp\big( {\frac{1}{k P^{1/s}}}\big)$, we have
$$M^{\mathcal{D}}_k(q, h) \gg_{s, k} qh^{k/2}P^{sk-s^2\frac{k}{2}}?$$
\subsection*{Notation}

Throughout the paper we use the symbol $\square$ as an abbreviation for the word ``square''.
For example, ``$a$ is a $\square$ modulo $q$'' reads ``$a$ is a square modulo $q$''.
Also, for functions $g(x)$ and $h(x)$, we use interchangeably Landau's and Vinogradov's notation
$g(x) = O(h(x))$, $g(x)\ll h(x)$ or $h(x) \gg g(x)$ to indicate that there exists a
constant $C > 0$ such that $|g(x)| \leq C|h(x)|$ for all $x$.
We use subscripts such as $\ll_{s,k}$ to indicate that the constant $C$ may depend on
parameters $s, k$.
We let $\phi$ denote the Euler's totient function, defined by $\phi(q)=\# \{1 \leq n \leq q : (n, q)=1 \}.$
We also write $P = \phi(q)/q$ and we let $+_{\mathbb Z}$
denote the addition in $\mathbb Z$, as opposed to modular addition.
 
\section{Main estimate}
 In this section we prove an exponential identity for the indicator function of $s$-tuples of reduced residues.
\begin{lemma}
\label{Lemma5}
Let $\mathcal{D}=\{h_1, \ldots , h_s\}$ be an admissible set. For square-free integers $q$ we have
$$k_q(m+h_1)\ldots k_q(m+h_s)= P_{_{\mathcal{D}}}\sum_{r|q}\frac{\mu(r)}{\phi_{\mathcal{D}}(r)}\sum_{\substack{a<r \\ (a, r)=1}} e\Big(m\frac{a}{r}\Big)\mu_{\mathcal{D}}(a, r),$$
where  $$\mu_{\mathcal{D}}(a, r)=\prod_{p|r}\bigg(\sum_{\substack{{s \in \mathcal{D}_p}}}e\Big(\frac{sa(r/p)_{p}^{-1}}{p}\Big)\bigg),$$
$\displaystyle{\phi_{\mathcal{D}}(r)=\prod_{p|r} (p-\nu_p(\mathcal{D}))},$ $$\displaystyle{ P_{\mathcal{D}}=\frac{\prod_{p|q} \big(p-\nu_p(\mathcal{D})\big)}{q}},$$  $(r/p)_{p}^{-1}$ is the inverse of $r/p$ in $\big(\mathbb{Z}/p\mathbb{Z}\big)^{*},$ and $\mathcal{D}_p$ consists of the reduction of elements of $\mathcal{D}$ modulo $p$.

\end{lemma}
\begin{proof}
The starting point in the method of Montgomery and Vaughan~\cite{mv}
is to use the following Fourier expansion of the indicator function of reduced residues:
$$k_q(m)=\sum_{r\mid q}\frac{\mu(r)}{r}\sum_{0\leq b < r}e\bigg(m\frac{b}{r}\bigg).$$
Using this expansion, we deduce that
\begin{align}
\notag k_q(m+&  h_1) \ldots k_q(m+ h_s) \\ & =\notag \sum_{r_1, r_2, \ldots , r_s \mid q} \frac{\mu(r_1) \ldots \mu(r_s)}{r_1\ldots r_s} \sum_{\substack{0< a_i  \leq r_i  \\ \sum_{i=1}^{s} \frac{a_i}{r_i}= \frac{a}{r}}}e\bigg(\sum_{i=1}^{s}(m+ h_i)\frac{a_i}{r_i}\bigg)\\ & \notag
=\sum_{\substack{r|q \\ a \leq r \\ (a, r)=1}}e\Big(m\frac{a}{r}\Big)\sum_{r_1, r_2, \ldots , r_s \mid q} \frac{\mu(r_1) \ldots \mu(r_s)}{r_1\ldots r_s}\sum_{\substack{0< a_i  \leq r_i  \\ \sum_{i=1}^{s} \frac{a_i}{r_i}= \frac{a}{r}}}e\bigg(\sum_{i=1}^{s} h_i\frac{a_i}{r_i}\bigg).
\end{align}
We fix $a, r$ and therefore it is enough to show that
\begin{align}
\label{eq5}
 &\sum_{r_1, r_2, \ldots , r_s \mid q} \frac{\mu(r_1) \ldots\mu(r_s)}{r_1\ldots r_s}
\displaystyle{\sum_{\substack{0< a_i  \leq r_i  \\ \sum_{i=1}^{s} \frac{a_i}{r_i}= \frac{a}{r}}}}e
\bigg(\sum_{i=1}^{s} h_i\frac{a_i}{r_i}\bigg) \\& \notag =P_{_{\mathcal{D}}}\frac{\mu(r)}{\phi_{\mathcal{D}}(r)} \prod_{p|r}
\bigg(\sum_{\substack{s \in \mathcal{D}_p}}e\Big(\frac{sa(r/p)_{p}^{-1}}{p}\Big)\bigg).
\end{align} To show this, note that we can write $$\frac{a}{r}\equiv \sum_{p|q}\frac{a_p}{p} \hspace{2 mm} \text{(mod }1\text{)}$$
uniquely where $0 \leq a_p <p$. Fixing $p_{0}|r$, we have that
\begin{equation}
\label{uniq}
\displaystyle{\frac{a}{r}\cdot\frac{r}{p_{0}}
\equiv \frac{a}{p_{0}}\equiv \frac{a_{p_{0}}}{p_0}\big(\frac{r}{p_0}\big) \hspace{2 mm}} \text{(mod }1\text{)},
\end{equation}
hence
$\displaystyle{a \equiv a_{p_{0} }\big(\frac{r}{p_{0}}\big) \hspace{2 mm} \text{(mod } p_{0} \text{)}}$.
Since $q$, and consequently $r$, are square-free, $\big(\frac{r}{p_0}, p_0\big)=1$, so for $a_{p_{0}} \neq 0$, we have that  $\displaystyle{a_{p_{0}} \equiv
a\big(\frac{r}{p_0}\big)^{-1}_{p_0} \hspace{2 mm} \text{(mod }p_{0} \text{)}}$.
Using \eqref{uniq}, we can can write the left hand side of the \eqref{eq5} in terms of the prime divisors of $q$.  Therefore \eqref{eq5} is equal to $$\displaystyle{\prod_{p|q} \sum_{q_i|p} \frac{\mu(q_1)\ldots \mu(q_s)}{q_1 \ldots q_s} \displaystyle{\sum_{\substack{0 \leq a_i < q_i \\ \sum \frac{a_i}{q_i}=\frac{a_p}{p}}}e\big(\sum_{i=1}^s h_i\frac{a_i}{q_i}\big)}}.$$
To simplify the condition $\sum \frac{a_i}{q_i}=\frac{a_p}{p}$, we write $$\displaystyle{\sum_{\substack{0 \leq a_i < q_i \\\sum \frac{a_i}{q_i}=\frac{a_p}{p}}}e\big(\sum_{i=1}^s h_i\frac{a_i}{q_i}\big)}= \sum_{v=1}^{p} \frac{1}{p} \displaystyle{\sum_{\substack{0 \leq a_i < q_i }}e\Big(\big(\frac{a_p}{p}-\sum_{i=1}^{s}\frac{a_i}{q_i}\big)v\Big)e\big(\sum_{i=1}^s h_i\frac{a_i}{q_i}\big)}.$$ Therefore \eqref{eq5} is equal to
\begin{align*}
&\displaystyle{\prod_{p|q} \sum_{q_i|p} \frac{\mu(q_1)\ldots \mu(q_s)}{q_1 \ldots q_s}\sum_{v=1}^{p} \frac{1}{p} \displaystyle{\sum_{\substack{0 \leq a_i < q_i }}e\Big(\big(\frac{a_p}{p}-\sum_{i=1}^{s}\frac{a_i}{q_i}\big)v\Big)e\big(\sum_{i=1}^s h_i\frac{a_i}{q_i}\big)}} \\ &=\displaystyle{\prod_{p|q} \sum_{v=1}^{p} \frac{e\big(v\frac{a_p}{p}\big)}{p}\prod_{i=1}^{s} \sum_{q_i|p} \frac{\mu(q_i)}{q_i}\displaystyle{\sum_{\substack{0 \leq a_i < q_i }}e\big(\frac{a_i}{q_i}(h_i -v)\big)}} \\ & = \displaystyle{\prod_{p|q} \sum_{v=1}^{p} \frac{e\big(v\frac{a_p}{p}\big)}{p}\prod_{i=1}^{s}}\Big( 1- \frac{1}{p}\sum_{a \leq p} e( \frac{a}{p}(h_i -v))\Big) \\ & = \displaystyle{\prod_{p|q} \sum_{\substack{v=1 \\ v \not \equiv h_i  \hspace{1 mm}\text{mod} \hspace{1 mm} p \\ h_i \in \mathcal{D} }}^{p} \frac{e\big(v\frac{a_p}{p}\big)}{p}}= \frac{\phi_{\mathcal{D}}(q)r}{q\phi_{\mathcal{D}}(r)}\prod_{p|r} \frac{\mu(p)}{p}\sum_{\substack{s \equiv h_i \text{mod} \hspace{1 mm} p \\ h_i \in \mathcal{D}\\ a_p \neq 0}} e\big(s \frac{a_p}{p}\big).
\end{align*}
The last equality holds since $a_p=0$ for $p\nmid r$, and for $a_p\neq 0$ we have that $$\sum_{\substack{v=1 \\ v \not \equiv h_i  \hspace{1 mm}\text{mod} \hspace{1 mm} p \\ h_i \in \mathcal{D} }}^{p} e\big(v\frac{a_p}{p}\big)=-\sum_{\substack{s  \in \mathcal{D}_p}} e\big(s \frac{a_p}{p}\big).$$
This completes the proof of the lemma.
\end{proof}
\section{Distribution of squares modulo $q$}
In this section we are going to prove Theorem \ref{th-0.2} and Corollary \ref{corr-01}. Before proceeding with the proof we derive a formula for the left hand side of \eqref{eq1.5}.  For $q$ square-free, $x$ is a square modulo $q$ if and only if $x$ is a square modulo $p$
for all primes $p$ dividing $q$. For each $p$ which divides $q$, let $\mathcal{D}_{p}:=\{h_{1,p}, \ldots , h_{\nu_p, p} \}$. By the Chinese Reminder Theorem there exists a set $\mathcal{D}=\{h_1, \ldots , h_s \} $, such that $\mathcal{D} \hspace{2 mm} \text{modulo} \hspace{2 mm} p $ is equal to $\mathcal{D}_{p}$, for all $p|q$. For instance let $h_1, h_2, \ldots h_s$ to be uniquely selected to satisfy the following congruences  $h_i \equiv h_{i,p} \text{ mod } p,$ for all $p|q$. In the case that $i> \nu_p $ and therefore $h_{i,p}$ does not exist, we take $h_{i,p}$ to be equal $h_{\nu_p,p}$. This explains how we can construct the set $\mathcal{D}$.\\
\\
\noindent Now if $$k_q(m+h_1)\ldots k_q(m+h_s)=1,$$
then $m \not \equiv -h_{i,p}$ modulo $p$, for $1 \leq i \leq \nu_{p}$ and for all $p$ dividing $q$. We now let
\begin{equation}
\label{eq5.1}
\mathcal{D}_{p}= \{-n_1, \ldots, -n_{\frac{p-1}{2}} \},
\end{equation}
 where $n_i$'s are quadratic non-residues modulo $p$. From $k_q(m+h_1)\ldots k_q(m+h_s)=1,$ it follows that $m$ is a square modulo $q$. Using Lemma \ref{Lemma5} we have that
\begin{align}
\label{eq5.3}
&\notag k_q(m+h_1)\ldots k_q(m+h_s)=\prod_{p|q} \frac{\frac{p+1}{2}}{p}\sum_{r|q}\frac{\mu(r)}{\prod_{p|r} \frac{p+1}{2}}\sum_{\substack{a \leq r \\ (a, r)=1}} e\Big(m\frac{a}{r}\Big)\mu_{\mathcal{D}}(a, r) \\ & = \frac{1}{2^{\omega(q)} P}\sum_{r|q}\frac{\mu(r)}{\prod_{p|r} \frac{p+1}{2}}\sum_{\substack{a \leq r \\ (a, r)=1}} e\Big(m\frac{a}{r}\Big)\mu_{\mathcal{D}}(a, r),
\end{align}
where $P=\frac{\phi(q)}{q}$. Summing this from $m=n+1$ to $n+h$ and then subtracting the term corresponding to $r=1$  we have
\begin{align}
\label{eq5.3.1}
&\notag \sum_{m=n+1}^{n+h}k_q(m+h_1)\ldots k_q(m+h_s)- \frac{h}{2^{\omega(q)}P}  \\ & = \frac{1}{2^{\omega(q)} P}\sum_{\substack{r|q \\ r>1}}\frac{\mu(r)}{\prod_{p|r} \frac{p+1}{2}}\sum_{\substack{a<r \\ (a, r)=1}} E\Big(\frac{a}{r}\Big)\mu_{\mathcal{D}}(a, r)e\Big(n\frac{a}{r}\Big),
\end{align}
 where
\begin{align}
 \notag & E(x)=\sum_{m=1}^{h} e(mx).
 \end{align}
 We square \eqref{eq5.3.1} and sum from $n=1$ to $q$ to obtain
\begin{align}
\label{eq5.3.2}
& \sum_{n=0}^{q-1}\left( \sum_{m=1}^{h}k_q(n+m+h_1)\ldots k_q(n+m+h_s)- \frac{h}{2^{\omega(q)}P} \right)^{2}=
\\ & \notag \frac{q}{4^{\omega(q)}P^2}\sum_{\substack{r_1 ,\hspace{1 mm} r_2 |q \\ r_1, r_2 >1}}\frac{\mu(r_1)\mu(r_2)}{\prod_{p|r_1} \frac{p+1}{2} \prod_{p|r_2} \frac{p+1}{2}} \displaystyle{ \sum_{\substack{a_i<r \\ (a_i, r_i)=1 \\ i=1, 2 \\ \frac{a_1}{r_1}+\frac{a_2}{r_2} \in \mathbb{Z}}}} E\Big(\frac{a_1}{r_1}\Big)E\Big(\frac{a_2}{r_2}\Big)\mu_{\mathcal{D}}(a_1, r_1)\mu_{\mathcal{D}}(a_2, r_2).
\end{align}
Now we are prepared to prove the Theorem \ref{th0-2}.
\begin{proof}[{\bf Proof of  Theorem \ref{th0-2}}] From the condition $\displaystyle{\frac{a_1}{r_1}+\frac{a_2}{r_2} \in \mathbb{Z}}$ in
\eqref{eq5.3.2} it follows that $r_1=r_2$ and $a_2=r-a_1$, thus we have
\begin{align}
\label{eq5.5}
 &\displaystyle{\sum_{n=0}^{q-1}\bigg( \sum_{\substack{m=1 \\ n+m \hspace{1 mm} \text{is } \square \hspace{1 mm}
 \text{mod}  \hspace{1 mm}q}}^{h} 1 - \frac{h}{2^{\omega(q)}P} \bigg)^{2}} = \frac{q}{4^{\omega(q)}P^2}\sum_{\substack{r |q \\ r >1}}\frac{4^{\omega(r)}}{ \prod_{p|r} (p+1)^{2}} \displaystyle{ \sum_{\substack{a<r \\ (a, r)=1 }}} \bigg| E\Big(\frac{a}{r}\Big)\mu_{\mathcal{D}}(a, r)\bigg|^2.
\end{align}
Now, we need to bound $\mu_{\mathcal{D}}(a, r)$. For each $n_i$ in $\mathcal{D}_p$ in \eqref{eq5.1}, employing the Legendre symbol
$$\bigg(\dfrac{-n_i a \big(r/p\big)^{-1}}{p}\bigg)=-\bigg(\dfrac{-1}{p}\bigg)\bigg(\dfrac{a}{p}\bigg) \bigg(\dfrac{\big(r/p\big)^{-1}}{p}\bigg).$$
Since $a\neq 0$ the sequence $\displaystyle{ \big\{-n_i a \big(r/p\big)^{-1}\big\} }$ is either the sequence of quadratic residues or the sequence of  quadratic non-residues modulo $p$. Using the Gauss bound for exponential sums over quadratic residues (respectively non-residues) \cite[Page 13]{Da}
\begin{equation}
 \bigg|\sum_{\substack{i}}e\Big(\frac{n_ia(r/p)_{p}^{-1}}{p}\Big)\bigg| =\begin{cases}
\frac{\sqrt{p}-1}{2}  &\text{if $\bigg(\dfrac{- a \big(r/p\big)^{-1}}{p}\bigg)=-1$},\\
\frac{\sqrt{p}+1}{2} &\text{otherwise,}
\end{cases}
\end{equation}
if $p \equiv 1$ modulo $4$ and
\begin{equation}
 \bigg|\sum_{\substack{i}}e\Big(\frac{n_ia(r/p)_{p}^{-1}}{p}\Big)\bigg| =\frac{\sqrt{p+1}}{2}
\end{equation}
if  $p \equiv 3$ modulo $4$. Consequently, for $a\neq 0$,
 \begin{equation}
 \label{eq5.6}
 \prod_{p|r} \frac{\sqrt{p}-1}{2} \leq |\mu_{\mathcal{D}}(a, r) | \leq \prod_{p|r} \frac{\sqrt{p}+1}{2}.
\end{equation}
Using this in \eqref{eq5.5} we have the upper bound
\begin{equation}
\label{up}
 \displaystyle{\sum_{n=0}^{q-1}\bigg( \sum_{\substack{m=1 \\ n+m \hspace{1 mm} \text{is } \square
\hspace{1 mm}  \text{mod}  \hspace{1 mm}q}}^{h} 1 - \frac{h}{2^{\omega(q)}P} \bigg)^{2}} \leq \frac{q}{4^{\omega(q)}P^2}
\sum_{\substack{r |q \\ r >1}} \prod_{p|r}\frac{\big(\sqrt{p}+1\big)^2}{  (p+1)^{2}} \displaystyle{ \sum_{\substack{a<r \\ (a, r)=1 }}}
 \bigg| E\Big(\frac{a}{r}\Big)\bigg|^2,
 \end{equation}
 and the lower bound
 \begin{equation}
 \label{down} \frac{q}{4^{\omega(q)}P^2}
\sum_{\substack{r |q \\ r >1}} \prod_{p|r}\frac{\big(\sqrt{p}-1\big)^2}{  (p+1)^{2}} \displaystyle{ \sum_{\substack{a<r \\ (a, r)=1 }}}
 \bigg| E\Big(\frac{a}{r}\Big)\bigg|^2 \leq \displaystyle{\sum_{n=0}^{q-1}\bigg( \sum_{\substack{m=1 \\ n+m \hspace{1 mm} \text{is } \square
\hspace{1 mm}  \text{mod}  \hspace{1 mm}q}}^{h} 1 - \frac{h}{2^{\omega(q)}P} \bigg)^{2}}.
\end{equation}
Using the bound(\cite[Lemma 4]{mv}),
\begin{equation}
\label{Ex}
\displaystyle{ \sum_{\substack{a<r \\ (a, r)=1 }}} \Big| E\Big(\frac{a}{r}\Big)\Big|^2
< r \min(r,h),
\end{equation} and by employing this bound in \eqref{up} we have
$$\displaystyle{\sum_{n=0}^{q-1}\bigg( \sum_{\substack{m=1 \\ n+m \hspace{1 mm} \text{is } \square \hspace{1 mm}  \text{mod}
 \hspace{1 mm}q}}^{h} 1 - \frac{h}{2^{\omega(q)}P} \bigg)^{2}}\leq \frac{q}{4^{\omega(q)}P}h\prod_{p|q}(2+\frac{2p^{3/2}-p-1}{p^2+2p+1})< \frac{q}{2^{\omega(q)}P}h\prod_{p|q}(1+\frac{1}{\sqrt p}). $$ For the
lower bound, let $r> h^2.$ Then we have
$$ \phi(r)h \ll \displaystyle{ \sum_{\substack{a<r \\ (a, r)=1 }}} \Big| E\Big(\frac{a}{r}\Big)\Big|^2 .$$ Therefore,
 $$  \frac{q}{4^{\omega(q)}P}h\sum_{\substack{r>h^2 \\ r|q}}\prod_{p|r} (1-\frac{3}{\sqrt{p}}) \ll \displaystyle{\sum_{n=0}^{q-1}\bigg( \sum_{\substack{m=1 \\ n+m \hspace{1 mm} \text{is } \square \hspace{1 mm}  \text{mod}
 \hspace{1 mm}q}}^{h} 1 - \frac{h}{2^{\omega(q)}P} \bigg)^{2}}.$$
\end{proof}
\begin{proof}[{\bf Proof of  Corollary \ref{corr-01}}] Let
$$L(x)=\# \bigg\lbrace i : 1 \leq i \leq \prod_{p|q}\big(\frac{p+1}{2}\big) \text{ and } s_{i+1}-s_{i}> x
\bigg\rbrace.$$
 Then \begin{equation}
 \label{Corr-proof}
\sum_{s_i < q} (s_{i+1}-s_i)^2 = 2 \int_{0}^{ \infty } L(y)y dy.
\end{equation}

\noindent For $y< 2^{\omega(q)}P^{-1} \log q \prod_{p|q}(1+\frac{1    }{\sqrt{p}})$ we bound \eqref{Corr-proof} trivially. To bound $L(y)$ we note that if
 $s_{i+1}-s_i> h , $ then $$\sum_{\substack{m=1 \\ n+m \hspace{1 mm} \text{is } \square \hspace{1 mm}  \text{mod}
\hspace{1 mm}q}}^{h} 1 - \frac{h}{2^{\omega(q)}P}= -\frac{h}{2^{\omega(q)}P}, $$
for $s_i \leq n \leq s_{i+1}-h.$ Therefore we have
\begin{equation}
\label{corr-prof-1}
 \sum_{s_{i+1}-s_i>h}^{} (s_{i+1}-s_i -h)\big(\frac{h}{2^{\omega(q)}P}\big)^2 \ll \displaystyle{\sum_{n=0}^{q-1}\bigg( \sum_{\substack{m=1 \\ n+m \hspace{1 mm} \text{is } \square \hspace{1 mm}  \text{mod}
 \hspace{1 mm}q}}^{h} 1 - \frac{h}{2^{\omega(q)}P} \bigg)^{2}}.
\end{equation}
  Now if we take $y=[h/2]$ then the left hand side of \eqref{corr-prof-1} is $$\gg L(y)y\big(\frac{y}{2^{\omega(q)}P}\big)^2.$$ Thus, by employing Theorem \ref{th0-2} we get the following bound: $$L(y) \ll \frac{2^{\omega(q)}P}{y^2}\prod_{p|q}(1+\frac{1    }{\sqrt{p}}).$$ Applying this bound in the integral in \eqref{Corr-proof} and the fact that for $y>q,$ $L(y)=0$ completes the proof of the Corollary.
 \end{proof}
\section{The general case}
In this section we will prove Theorems \ref{th-0.2},  \ref{coro1} and \ref{Thm-M_p}. Let $\Omega_p \subset \mathbb{Z}/p\mathbb{Z}$. We are interested in numbers less than $q$ such that, modulo $p$, they do not occupy any congruence classes in $\Omega_p$, i.e.\ $\{m \leq q : m \notin \Omega_p \text{ mod } p\}$. By the Chinese Remainder Theorem there exist $\prod_{p|q} (p-|\Omega(p)|)$ such numbers. A natural question is to ask about their distribution modulo $q$ (see \cite{GK}). Lemma \ref{Lemma5} shows the connection between the distribution of
these numbers and the exponential sum over elements in $\Omega_p$. Let $\mathcal{D}=\{h_1, \ldots , h_s \} $ be a set such that $\mathcal{D}_p =\{-\omega : \omega \in \Omega_p\}$. If $k_q(m+h_1)\ldots k_q(m+h_s)=1$, then $m$ is not congruent to any member of $\Omega_p$ modulo $p$. Now we take a look at the distribution of these numbers. Observe that
\begin{align}
\notag \sum_{m=1}^{h}& k_q(m+h_1)\ldots k_q(m+h_s) \\& =\prod_{p|q} \frac{p-|\Omega_p |}{p}\sum_{r|q}\displaystyle{\frac{\mu(r)}{\prod_{p|r} (p-|\Omega_p |)}}\sum_{\substack{a<r \\ (a, r)=1}} E\Big(\frac{a}{r}\Big)\mu_{\mathcal{D}}(a, r).
\end{align}
By a calculation similar to \eqref{eq5.3.2} we have
\begin{align}
\label{Th-4p}
 &\ \sum_{n=0}^{q-1}\left( \sum_{m=1}^{h}k_q(n+m+h_1)\ldots k_q(n+m+h_s)- h\prod_{p|q} \bigg( \frac{p-|\Omega_p |}{p}\bigg) \right)^{2}  \\&= q\prod_{p|q} \bigg( \frac{p-|\Omega_p |}{p}\bigg)^2 \sum_{ \substack{r \mid q \\ r>1}}  \frac{1}{\prod_{p|r} (p-|\Omega_p|)^2} \sum_{\substack{0 < a \leq r \\ (a, r)=1}}    \left |  E\bigg( \frac{a}{r}\bigg)\mu_{\mathcal{D}}(a, r) \right |^2.\notag
\end{align}
In the next lemma we bound the variance.
\begin{lemma}
\label{t4.1}
Assume that for each $p|q$, $|\Omega_p| = c^{\prime}_p p$ with $(p-| \Omega_p|> p^{1/2 + \epsilon})$,and $|\mu_{\mathcal{D}}(a, p)| < c_{p} \sqrt{p},$ where $\displaystyle{c^{\prime}_p}  <1$. Then we have that $$
\sum_{n=0}^{q-1}\Bigg( \sum_{\substack{m \in [n, n+h] \\ m \notin \Omega_p  \text{ mod } p \\ \forall p|q}}1- h\prod_{p|q} \bigg( \frac{p-|\Omega_p|}{p}\bigg) \Bigg)^{2}  \leq qh\prod_{p|q} \bigg((1-c^{\prime}_p)^2 +c^{2}_{p} \bigg).$$
\end{lemma}
\begin{proof}
Using the assumptions in Lemma \ref{t4.1} and \eqref{Ex} we have
\begin{align*}
 &\ \sum_{n=0}^{q-1}\left( \sum_{m=1}^{h}k_q(n+m+h_1)\ldots k_q(n+m+h_s)- h\prod_{p|q} \bigg( \frac{p-|\Omega_p |}{p}\bigg) \right)^{2}  \\& \ll q\prod_{p|q} \bigg( \frac{p-c^{\prime}_p p}{p}\bigg)^2 \sum_{ \substack{r \mid q \\ r>1}}  \frac{h\prod_{p|r} (c_p p)^2}{\prod_{p|r} (p-c^{\prime}_p p)^2} = qh\prod_{p|q} (1-c^{\prime}_p)^2 \sum_{ \substack{r \mid q \\ r>1}}  \prod_{p|r}\bigg(\frac{ c_p}{ 1-c^{\prime}_p }\bigg)^2 \\ & < qh\prod_{p|q} (1-c^{\prime}_p)^2  \Bigg(1+\bigg(\frac{ c_p}{ 1-c^{\prime}_p }\bigg)^2 \Bigg)=qh\prod_{p|q} \bigg((1-c^{\prime}_p)^2 +c^{ 2}_{p} \bigg).
\end{align*}
This completes the proof of the lemma.\\
\end{proof}
\noindent
\begin{proof}[{\bf Proof of Theorem \ref{th-0.2}}]
This follows from Lemma \ref{t4.1} by taking $c^{\prime}_p=(p-1)/2p.$  Recall that $c^{\prime}_p=\frac{|\Omega_p|}{p}.$
\end{proof}

\noindent Next we prove Theorem \ref{coro1}:
\begin{proof}[{\bf Proof of Theorem \ref{coro1}}].
Let $\mathcal{D}^{*}=\{h_1, \ldots , h_s \} $ be an admissible set such that $\mathcal{D}^{*}_p= -\Omega_p=\{0, -2, \ldots , -(p-1)\}$. Let $\displaystyle{\frac{a}{r}=\sum_{p|r} \frac{a_p}{p}}$,
 for $\displaystyle{a_p=\frac{p\pm 1}{2}}$. Since $\mathcal{D}^{*}_p=\{0, 2, \ldots , p-1\}$,  applying Lemma \ref{Lemma5} we have that \\
 \begin{equation}
 \label{eq10}
|\mu_{\mathcal{D}^{*}}(a, r)|= \prod_{p|r}\Bigg| \sum_{s\in \mathcal{D}^{*}_p} e\Big(\frac{sa_p}{p}\Big) \Bigg|=\prod_{p|r}
 \Bigg|\frac{e\big(\frac{1}{2p}\big)+1}{e\big(\frac{1}{p}\big)-1}  \Bigg| \geq \prod_{p|r}  \frac{p}{\pi}.
\end{equation}
\\
\noindent Here, similar to the square case (section 2), we have $\displaystyle{P_{\mathcal{D}^{*}}= \frac{1}{2^{\omega(q)}}P}$ and
$\displaystyle{ \phi_{\mathcal{D}^{*}}(r)=\prod_{p|r} \frac{p-1}{2}}.$ Consequently, using \eqref{eq10} we have, similarly to \eqref{eq5.3} and \eqref{eq5.5}, that
\begin{align}
\label{eq 12}
 &\ \notag \sum_{n=0}^{q-1}\left( \sum_{m=1}^{h}k_q(n+m+h_1)\ldots k_q(n+m+h_s)-  \frac{hP}{2^{\omega(q)}} \right)^{2}  \\&= \notag
\frac{qP^2}{2^{2\omega(q)}} \sum_{ \substack{r \mid q \\ r>1}} \Bigg( \frac{1}{ \Big(\prod_{p|r} \frac{p-1}{2}\Big)^2} \Bigg)
\sum_{\substack{  0 < a \leq r \\ (a, r)=1    }}  \Big| E\bigg( \frac{a}{r}\bigg)\mu_{\mathcal{D}^{*}}(a, r)\Big|^2 \\ & \geq \notag  \frac{qP^2}{2^{2\omega(q)}} \sum_{ \substack{r \mid q \\ r>1}}  \frac{4^{\omega(r)}}{ \phi(r)^2}  \sum_{\substack{ a_p= \frac{p\pm 1}{2}  \\ p|r   }}  \Big| E\bigg( \displaystyle{\sum_{p|r} \frac{a_p}{p} }\bigg)\mu_{\mathcal{D}^{*}}(a, r)\Big|^2 \\ & \geq  \frac{qP^2}{2^{2\omega(q)}} \sum_{ \substack{r \mid q \\ r>1}}  \frac{4^{\omega(r)}r^2}{ \phi(r)^2 \pi^{2\omega(r)}}  \sum_{\substack{ p|r}} \Big| E\bigg( \displaystyle{\sum_{p|r} \frac{1}{2}\pm \frac{1}{2p} }\bigg)\Big|^2.
\end{align}
Now, for $r$ with an even number of distinct prime factors and $ \displaystyle{\parallel \sum_{p|r} \frac{\pm1}{p}  \parallel \ll 1/h}$, where $\parallel \cdot \parallel$ denotes the distance to the nearest integer, we have 
$$ \Big| E\bigg( \displaystyle{\sum_{p|r} \frac{1}{2} \pm \frac{1}{2p} }\bigg)\Big|^2 \gg h^2.$$ Consequently  \eqref{eq 12} is $$ \gg \frac{qP^2}{2^{2\omega(q)}}h^2 \sum_{ \substack{r \mid q \\ r>1 \\ \parallel \sum_{p|r} \frac{\pm1}{p}  \parallel \ll 1/h }}  \frac{4^{\omega(r)r^2}}{ \phi(r)^2 \pi^{2\omega(r)}} . $$
Now let $r=p_1p_2$, with $\displaystyle{a_{p_{1}}=\frac{p_1+ 1}{2}}$ and $\displaystyle{a_{p_{2}}=\frac{p_2 - 1}{2}},$ we have $$\parallel  \frac{p_1+ 1}{2p_1}+\frac{p_2- 1}{2p_2}\parallel   =\parallel \frac{1}{2p_1}-\frac{1}{2p_2}\parallel \ll \Big| \frac{\log X}{X^2}\Big| \ll \frac{1}{h},$$ which implies that \eqref{eq 12} is $$\gg \frac{qP^2}{2^{2\omega(q)}}h^2.$$
\end{proof}
\begin{remark}
 We picked $h= \frac{X^2}{\log X }$, so that the expectation of $$\# \{m \in (n, n+h] : m \not \in \mathcal{D}^{*}_p \text{  mod } p,   \text{ for all }\hspace{2 mm} p|q \}=\displaystyle{h\prod_{p|q}\frac{p+1}{2p}}=\frac{X^2P}{2^{\lfloor \log X\rfloor} \log X}$$ is greater than $1$. This is important in order to have the possibility of cancellation inside $$\sum_{n=0}^{q-1}\left( \sum_{m=1}^{h}k_q(n+m+h_1)\ldots k_q(n+m+h_s)-  \frac{hP}{2^{\omega(q)}} \right)^{2}.$$
\end{remark}
\noindent We complete this section with the proof of Theorem \ref{Thm-M_p}.
\begin{proof}[{\bf Proof of  Theorem \ref{Thm-M_p}}] We begin with giving the proof for equation \eqref{th4p0}. Recall that \begin{equation*}
 \Omega^{\prime}_{K, p}:= \big \{ x+ x^{-1} : 1\leq x, x^{-1} \leq p-1 \text{ and } x+_{\mathbb Z} x^{-1} <p\big \},
\end{equation*}
If $x +_{\mathbb Z} x^{-1}<p$ then $(p-x) + _{\mathbb Z} (p-x)^{-1} \geq p,$ Therefore half of the congruence classes modulo $p$ contribute to the size of $\Omega^{\prime}_{K, p}$.  Also, for $y<p$ we have $x +_{\mathbb Z} x^{-1}=x^{-1} +_{\mathbb Z}x =y.$ This means that each $y \in \Omega^{\prime}_{K, p} $ has a double multiplicity, with the exception of $y$ equal to $1+ 1^{-1}$. Considering the fact that for $p \equiv 1$ mod $4,$ there exists an $x$ such that $x^{-1}=p-x$, and therefore $x+x^{-1}=p.$ This completes the proof of equation \eqref{th4p0}. \\

\noindent Now let $\Omega_p:= \big \{-\omega : \omega \in \{0, 1, \ldots, p-1\}\setminus \Omega^{\prime}_{K, p}\big \} =  \{\omega_1, \ldots, \omega_{|\Omega_p|}\}.$
Using \eqref{th4p0} we have $|\Omega_p|=\frac{3}{4}p + O(\frac{1}{p}).$
If $k_p(m+\omega_1)k_p(m+\omega_2)\ldots k_p(m+\omega_{|\Omega_p|})=1,$ then $m \in \Omega_{K, p}. $ We use \eqref{Th-4p} to transform the left hand side of \eqref{th4p1}, and we have \begin{equation}
\frac{1}{p}\sum_{n=0}^{p-1}\Bigg( \sum_{\substack{m \in [n, n+h] \\ m \in \Omega^{\prime}_{K, p} }}1- \frac{h}{p}|\Omega^{\prime}_{K, p}|
  \Bigg)^{2} \gg \frac{1}{p^2} \sum_{\substack{0 < a \leq p-1 }}    \left |  E\bigg( \frac{a}{p}\bigg)\mu_{_{\Omega_p}}(a, p) \right |^2.
\end{equation}
To finish the proof of the theorem it is enough to show that  $\big|  E\big( \frac{1}{p}\big)\mu_{_{\Omega_p}}(1, p) \big  | \gg hp.$ Since $h<p/2$ we have
that $\big|  E\big( \frac{1}{p}\big) \big  | \gg h.$ For $\mu_{_{\Omega_p}}$ we have
 $$\mu_{_{\Omega_p}}(1, p)=\sum_{\substack{x \in \Omega_p}}e\Big(\frac{x}{p}\Big).$$ Recall that $\Omega_{K,p} , \Omega^{\prime}_{K, p}$ are defined by \eqref{Omega_{K,p}} and \eqref{M_p3}.
Therefore if $-\omega \in \Omega_p$ then $\omega \in \big( \mathbb{Z}/p\mathbb{Z} \setminus \Omega_{K,p} \big) \bigcup \big
\{ x+ x^{-1} :  x+_{\mathbb Z} x^{-1} \geq p\big \}. $ We have
 $$\mu_{_{\Omega_p}}(1, p)=\sum_{\substack{x \in \mathbb{Z}/p\mathbb{Z} \setminus \Omega_{K,p}}}e\Big(\frac{-x}{p}\Big) +
 \sum_{\substack{\omega \in  \{ x+ x^{-1} :  x+_{\mathbb Z} x^{-1} \geq p \}}} e\Big(\frac{-\omega}{p}\Big).$$ Using Weil's bound for Kloosterman sums, the first sum above is $O(\sqrt p).$ For the second sum we prove
\begin{equation}
\label{Fourier}
\sum_{\substack{\omega \in  \{ x+ x^{-1} :  x+_{\mathbb Z} x^{-1} \geq p \}}} e\Big(\frac{-\omega}{p}\Big)= \frac{ i p}{2\pi} + O( \sqrt{p}\log{p} ).
\end{equation}
The following argument for \eqref{Fourier} was given by Will Sawin and Noam Elkies on Math Overflow \cite{Ar}. The left hand side of \eqref{Fourier} is equal to
$$\sum_{1 \leq x,y \leq p-1}\mathbf 1_{\{xy=1\} }  e(\frac{-x-y}{p}) \mathbf 1_{\{x+_{_{\mathbb Z}}y \geq p\}}.$$
We use a two dimensional Fourier transform to evaluate the left hand side of \eqref{Fourier}. Let $ \^{A}(a, b)$ be the Fourier transform of
 $\mathbf 1_{\{xy=1\} }$ and $\hat{B}(a, b)$ be the Fourier transform of $\mathbf 1_{\{x+y>p\}} e(\frac{x+y}{p})$.
Then by using Parseval-Plancherel formula, the sum in \eqref{Fourier} is:
\begin{equation}
\label{parsev}
\frac{\sum_{0 \leq a, b \leq p-1} \^{A}(a, b)  \overline{\hat{B}(a, b)} }{p^2}, 
\end{equation}
where
\begin{align*}
& \^{A}(a, b)= \sum_{0 \leq x <p} e\big(\frac{ax+ bx^{-1}}{p} \big)= S(a, b; p),
 \\ & \hat{B}(a, b)= \sum_{\substack{0\leq x,y<p\\ x+y>p}} e\big(\frac{(a-1)x + (b-1)y}{p}\big).
\end{align*}
Note that $\^{A}(a, b)$ is the Kloosterman sum unless $a=b=0$. For $\hat{B}(a, b)$ when $b \neq 1$ we have
\begin{eqnarray*}
 \hat{B}(a, b)  &=& \sum_{1\leq x <p} e\big(\frac{(a-1)x}{p}\big) \bigg( \sum_{p+1-x \leq y \leq p-1} e\big(\frac{(b-1)y}{p}\big) \bigg) \\ & =&\sum_{1 \leq x <p} e\big(\frac{(a-1)x}{p}\big) \frac{ e(b-1) - e(\frac{(b-1) (1-x)}{p})}{e(\frac{b-1}{p})-1}\\&=& \sum_{1 \leq x <p} \bigg(\frac{e(\frac{(a-1)x}{p})}{e(\frac{b-1}{p})-1}- \frac{ e(\frac{(a-b)x+b-1 )}{p})}{e(\frac{b-1}{p})-1}\bigg).
\end{eqnarray*}
The first term in the latter sum is  $\frac{p-1}{e((b-1)/p)-1)}$ if $a = 1$ and $\frac{- 1}{e((b-1)/p)-1}$ otherwise.  The second term in the latter sum is $\frac{(p-1) e((b-1)/p)}{e((b-1)/p)-1}$  if $a=b$ and $\frac{- e((b-1)/p)}{e((b-1)/p)-1)}$ otherwise.
\\
\\
\noindent Note that if $a=b=1$, then ${\hat{B}(1, 1)}$ is  $(p-1)(p-2)/2$. Also if $b=1$ and $a \neq 1$ we have $$\hat{B}(a, b)=\frac{p}{e(\frac{a-1}{p})-1}+1 \ll \frac{p^2}{a}.$$ Now the main term in \eqref{parsev} comes from the contribution of
$\^{A}(0, 0)\overline{\hat{B}(0, 0)}.$ The error term can be handled by using the Weil bound on ${\hat{A}(a, b)}$ for $(a, b) \neq (0, 0)$ and the above elementary estimates for $\overline{\hat{B}(a, b)}$ for $(a, b) \neq (1, 1).$
\end{proof}
\section{Higher central moments of reduced residues modulo $q$}
In this section we will improve the result in \cite{F.A} regarding the higher central moments of $s$-tuples of reduced residues. The improvement comes from using Lemma \ref{Lemma5} to transform characteristic functions of $s$-tuples of reduced residues to an expression in terms of exponential sums. The rest of the proof will follow Montgomery and Vaughan's \cite{mv} arguments (Lemma 7 and 8 in \cite{mv}). The important part of the proof is to estimate the innermost sum in \eqref{eq6}, which we divide into two cases: diagonal and non-diagonal configurations. In the diagonal configuration the estimate derived is good enough for our purposes. In the non-diagonal configuration we use Lemma 7 and 8 in \cite{mv} to save a small power of $h$.
 Let $\mathcal{D}=\{h_1, \ldots , h_s \} $ be a fixed admissible set. By employing Lemma \ref{Lemma5} we have that
\begin{align*}
 &\ \sum_{n=0}^{q-1}\left( \sum_{m=1}^{h}k_q(n+m+h_1)\ldots k_q(n+m+h_s)- hP_\mathcal{D} \right)^{k}  \\&= qP_\mathcal{D}^{k}
\sum_{ \substack{r_i \mid q \\ r_i>1}} \left( \prod_{i=1}^{k} \frac{\mu(r_{i})}{\phi_{\mathcal{D}}(r_{i})}\right)
\sum_{\substack{  0 < a_{i} \leq r_{i} \\ (a_{i}, r_{i})=1   \\ \sum_{i=1}^{k} \frac{a_{i}}{r_{i}} \in \mathbb{Z} }} \left(  E\bigg( \frac{a_{1}}{r_{1}}\bigg)\mu_{\mathcal{D}}(a_1, r_1)\ldots E\bigg( \frac{a_{k}}{r_{k}}\bigg)\mu_{\mathcal{D}}(a_k, r_k) \right),
\end{align*}
where $$\mu_{\mathcal{D}}(a, r)=\prod_{p|r}\bigg(\sum_{\substack{s \equiv h_i \text{mod} \hspace{1 mm} p \\ h_i \in \mathcal{D}}}e\Big(\frac{sa(r/p)_{p}^{-1}}{p}\Big)\bigg).$$
Let $F(x)=\min(h, \frac{1}{\Vert x \Vert})$ where $\Vert x \Vert$ is the distance between $x$ and the closest integer to $x$. We have that $|E(x)|\leq F(x)$.
Since $|\mu_{\mathcal{D}}(a, r)| \leq s^{\omega(r)}$ we have
\begin{equation}
\label{eq6}
M^{\mathcal{D}}_k(q, h) \ll  qP_\mathcal{D}^{k} \sum_{ \substack{r \mid q }} \sum_{\substack{r_i|r \\ r_i >1 \\ [r_1, \ldots r_k]=r}}
\prod_{i=1}^{k} \frac{s^{\omega(r_i)}}{\phi_{\mathcal{D}}(r_{i})} \sum_{\substack{  0 < a_{i} \leq r_{i} \\ (a_{i}, r_{i})=1   \\
\sum_{i=1}^{k} \frac{a_{i}}{r_{i}} \in \mathbb{Z} }}  F\bigg( \frac{a_{1}}{r_{1}}\bigg)\ldots F\bigg( \frac{a_{k}}{r_{k}}\bigg).
\end{equation}
\begin{proof}[\bf {Proof of Theorem \ref{thm4}}]
We use the method in \cite{mv} to bound $$\sum_{\substack{  0 < a_{i} \leq r_{i} \\ (a_{i}, r_{i})=1   \\ \sum_{i=1}^{k}
\frac{a_{i}}{r_{i}} \in \mathbb{Z} }}  F\bigg( \frac{a_{1}}{r_{1}}\bigg)\ldots F\bigg( \frac{a_{k}}{r_{k}}\bigg)$$ in  \eqref{eq6}.
 First we focus on diagonal configuration i.e. $r_{1}=r_{2}, r_{3}=r_{4},  \ldots, r_{k-1}=r_k$ and $r_2, r_4, \ldots, r_k$ are relativity co-prime.
 In the diagonal configuration we have that
\begin{align*}
& \sum_{\substack{  0 < a_{i} \leq r_{i} \\ (a_{i}, r_{i})=1    \\ \sum_{i=1}^{k}  \frac{a_{i}}{r_{i}} \in \mathbb{Z} }}
 F\bigg( \frac{a_{1}}{r_{1}}\bigg)\ldots F\bigg( \frac{a_{k}}{r_{k}}\bigg) \leq \sum_{\substack{  0 < a_{1} \leq r_{1} }}
 F\bigg( \frac{a_{1}}{r_{1}}\bigg)^2 \ldots \sum_{\substack{  0 < a_{k-1} \leq r_{k-1} }}F\bigg( \frac{a_{k-1}}{r_{k-1}}\bigg)^{2} \\ &  \leq r_1 r_3 \ldots r_{k-1}h^{k/2}= [r_1 r_3 \ldots r_{k-1}]h^{k/2}.
\end{align*}
Consequently, the contribution of the the diagonal configuration in \eqref{eq6} is less than
 \begin{align}
\label{4.4.4.4}
\notag & qP_\mathcal{D}^{k} \sum_{r|q} \sum_{ \substack{[r_1, ..., r_k]=r \\ (r_{2i-1}, r_{2j-1})=1 \\ i\neq j }}
\left( r_1\frac{s^{2\omega(r_1)}}{\phi_{\mathcal{D}}(r_{1})^2}\right) \left( r_3\frac{s^{2\omega(r_3)}}{\phi_{\mathcal{D}}(r_{3})^2}\right)
\ldots \left( r_{k-1}\frac{s^{2\omega(r_{k-1})}}{\phi_{\mathcal{D}}(r)^2}\right)h^{k/2} \\ & = \notag qP_\mathcal{D}^{k} \sum_{r|q}
\left( r\frac{(s^{2} \frac{k}{2})^{\omega(r)}}{\phi_{\mathcal{D}}(r)^2}\right)h^{k/2} =qP_\mathcal{D}^{k}  \prod_{p|q}
\bigg(1+\frac{ps^2 \frac{k}{2}}{\big(p-\nu_p(\mathcal{D})\big)^2}\bigg)  h^{k/2}  \\ & \ll qh^{k/2}  P^{sk-s^2\frac{k}{2}}.
 \end{align}
In \eqref{4.4.4.4} we used the fact that the number of $k$-tuples $(r_1, \ldots, r_k)$ with $[r_1, \ldots, r_k]=r$ such that each $p$ divides exactly two of $r_i$
is less than $(k/2)^{\omega(r)}$ (see \cite{mv-bas}). In the non-diagonal configuration Lemma 7 in \cite{mv} allows us to save a small power of $h$. Now we state the Lemma 7 in \cite{mv} and explain how it should be apply. Our aim is to get the following
 \begin{equation}
  \label{eq9}
  M^{\mathcal{D}}_k(q, h) \ll q h^{k/2} P^{sk-s^2\frac{k}{2}}\big(1+ h^{-\frac{1}{7k}}P^{-(s+1)^{k}}\big).
 \end{equation}
 This bound is analogous to \cite[Lemma 8]{mv} and its proof is nearly identical. The key difference is in \eqref{eq6} we have $s^{\omega(r)}/\phi_{\mathcal{D}(r)}$ instead of $1/\phi(r)$. Our main tool is the following lemma.
\begin{lemma}[Montgomery and Vaughan]
\label{lemm-v}
For $k \geq 3$, let $r_1,\ldots, r_k$ be square free numbers with $r_i \geq 1$ for all $i$. Further let $r = [r_1, r_2,..., r_k]$, $d = (r_l, r_2)$, $r_1 = dr^{\prime}_1, r_2= dr^{\prime}_2$. and write $d = st$ where $s|r_3 \ldots r_k$, $(t,r_3r_4 \ldots r_k)= 1$. Then
\begin{equation}
\label{T}
\sum_{\substack{  0 < a_{i} \leq r_{i} \\ (a_{i}, r_{i})=1   \\
\sum_{i=1}^{k} \frac{a_{i}}{r_{i}} \in \mathbb{Z} }}  F\bigg( \frac{a_{1}}{r_{1}}\bigg)\ldots F\bigg( \frac{a_{k}}{r_{k}}\bigg) \ll r_1 \ldots r_k r^{-1}(T_1+T_2+T_3+T_4)
\end{equation}
where
\begin{align*}
 & T_1=h^{-1/20}; \\ & T_2=d^{-1/4} \text{when } r_i > h^{8/9} \text{for all } i, \\ & T_2=0 \text{  otherwise;} \\ & T_3 =s^{-1/2} \text{when } r_i > h^{8/9} \text{ for all } i \text{and } r_1=r_2, \\ & T_3=0 \text{  otherwise;}
\end{align*}
and
\begin{align*}
  \hspace{15 mm}& T_4=\Bigg(\frac{1}{r_1r_2sh^2} \sum_{(\tau, t)=1}F\bigg(\frac{\parallel r^{\prime}_1s\tau \parallel}{r^{\prime}_1s}\bigg)^2F\bigg(\frac{\parallel r^{\prime}_2s\tau \parallel}{r^{\prime}_2s}\bigg)^2\Bigg)^{1/2} \\ & \text{ when } h^{8/9} < r_i \leq h^2 \text{ for }  i = 1,2, t > d^{1/2} \text{and } d < h^{5/9}, \\ & T_4=0 \text{  otherwise.}
\end{align*}
\end{lemma}
We shall also use the following estimate \cite[Lemma 1]{mv} 
\begin{equation}
\label{fund}
\sum_{\substack{  0 < a_{i} \leq r_{i} \\ (a_{i}, r_{i})=1   \\
\sum_{i=1}^{k} \frac{a_{i}}{r_{i}} \in \mathbb{Z} }}  F\bigg( \frac{a_{1}}{r_{1}}\bigg)\ldots F\bigg( \frac{ a_k}{r_{k}}\bigg) \ll \frac{1}{r}\prod_{i=1}^{k}\Bigg(r_i \sum_{(a_i, r_i)=1} F\bigg(\frac{a_i}{r_i}\bigg)^2\Bigg)^{1/2}. 
\end{equation}
Now we explain how to choose $r_1, r_2$ in order to apply Lemma \ref{lemm-v}. Note that we only need to consider those $k$-tuples $r = (r_1, r_2,, ..., r_k)$ for which $r_i > 1, [r_1,..., r_k] = r$, and each prime divisor of $r$ divides at least two of the $r_i$, since otherwise the sum on the left hand side of \eqref{T} is empty. If $r_i < h^{8/9}$ for some $i,$ then by using \eqref{fund} and \cite[Lemma 4]{mv} we have our desired result. Now suppose that $r_i > h^{8/9}$  for all $i,$ and  set $d_{ij} = (r_i, r_j)$. For each $i$ we can find a $j$, such that
\begin{equation}
\label{12}
 d_{i,j} \geq h^{8/(9k-9)}.
\end{equation}
If there is a pair $(i, j)$ for which this holds and $r_i \neq r_j$, then in Lemma \ref{lemm-v} we choose these to be $r_1, r_2$. We note that if $r_i = r_1$ then $d_{i,j} = r_i > h^{8/9}$, and \eqref{12} holds. Suppose now that \eqref{12} holds only when $r_i = r_j$. If there is a triple $(i, j, k)$ such that $r_i = r_1 =r_k$, then we apply  Lemma \ref{lemm-v} with  $r_i, r_j$ as $r_1, r_2$. Otherwise the $r_i$ are equal in distinct pairs, say $r_1 = r_2, r_3 = r_4, \ldots , r_{k-l} = r_k,$ and $k$ is even. Let $v$ be the product of all those prime factors of $r$ which divide more than one of the numbers $r_2, r_4, r_6,\ldots , r_k.$ Then there exists $i$ such that
\begin{equation}
\label{v-}
\bigg(r_{2i}, \prod_{j\neq i}r_{2j}\bigg)   \geq v^{4/k}.                                                                                                                                                                                                                                                                                                                                                                                                                                                                                                                                                                                                                                                                                                                                                                                                                                                                                                                                                                                                                                                                                                                                                                                                                                                                                                                                                                  \end{equation}
In this case we take $r_1$ and $r_2$ to be $r_{2i-1}, r_{2i}$ and by employing Lemma \ref{lemm-v} we have
\begin{equation}
 \label{eq6-1}
M^{\mathcal{D}}_k(q, h) \ll  qP_\mathcal{D}^{k} \sum_{ \substack{r \mid q }} \frac{1}{r} \sum_{\substack{r_i|r \\ r_i >1 \\ [r_1, \ldots r_k]=r}}
\prod_{i=1}^{k} \frac{s^{\omega(r_i)}r_i}{\phi_{\mathcal{D}}(r_{i})} \big(T_1+T_2+T_3+T_4\big).
\end{equation}
Note that if any of $T_2$, $T_3$, or $T_4$ is non-zero then $d \geq h^{8/(9k-9)}.$ The contribution of $T_1$ to \eqref{eq6-1} is
\begin{equation}
\label{5.9}
\ll qP^{sk}\prod_{p|q}\bigg(1+\frac{\big(1+sp/\phi_{\mathcal{D}}(p)\big)^k}{p}\bigg) \ll qP^{sk-(s+1)^k}h^{-1/20}. 
\end{equation}
By the selection of $r_1, r_2$ we have that if $T_2 \neq 0$ then $d \geq h^{8/(9k-9)}$. Therefore the contribution of $T_2$ to \eqref{eq6-1} is $\ll$
$ qP^{sk-(s+1)^k}h^{-2/(9k-9)}.$ Now for $T_3$ we have $r_1 = r_2$ and $r_1 \geq h^{8/9}$. If $r_1 = r_2 = r_i$ for some $i > 2$, then $s = r > h^{8/9}$, so that $T_3 < T_1$ and therefore the contribution of such $T_3$  to \eqref{eq6-1} is smaller than $T_1$ . It remains to consider the case when $r_1 = r_2, r_3 = r_4,..., r_{k- 1} = r_k.$ Let $r = uv$ where $u$ is the product of those primes dividing exactly one of $r_2, r_4,..., r_k$. Then each prime divisor of $v$ divides two or more of the $r_{2i}$. By our choice of $r_1, r_2$ we have $s\geq v^{4/k}$. Put $r_i = u_iv_i$ where $u_i = (r_i, u)$ and $v_i = (r_i, v)$. Suppose that $u$ and $v$ are fixed, and let $C(u, v)$ denote the set of $(r_1, \ldots, r_k)$ of the sort under consideration. We have $|C(u, v)| \leq d_{k/2}(u)d(v)^{k/2}$. Using the change of variable $r=uv$ and by rearranging the sum in \eqref{eq6-1}, for the contribution of $T_3$  we have
\begin{align}
\label{5.10}
 \sum_{uv|q}\frac{1}{uv} \sum_{(r_1, \ldots, r_k) \in C(u, v)}  \Big(\prod \frac{s^{\omega(r_i)}r_i}{\phi_{\mathcal{D}}(r_{i})}\Big)T_3 \ll & \sum_{uv|q}\frac{d_{k/2}(u)\bigg(\frac{s^{\omega(u)}u}{\phi_{\mathcal{D}}(u)} \bigg)^2 d(v)^{k/2}\bigg(\frac{s^{\omega(v)}v}{\phi_{\mathcal{D}}(v)}\bigg)^k}{uv^{1+2/k}} \\ \notag & = \prod_{p|q}\Bigg(1+  \frac{k s^2p}{2\phi_{\mathcal{D}}^2(p)}+ \frac{2^{k/2}\big(sv/\phi_{\mathcal{D}}(p)\big)^k}{p^{1+2/k}} \Bigg)
\\ & \ll \notag P^{-\frac{s^2k}{2}}.
\end{align}
For the contribution of $T_4$, by the Cauchy inequality, we have
\begin{equation}
\label{5.11}
 \sum_{ \substack{r \mid q }} \frac{1}{r} \sum_{\substack{r_i|r \\ r_i >1 \\ [r_1, \ldots r_k]=r}}
\prod_{i=1}^{k} \frac{s^{\omega(r_i)}r_i}{\phi_{\mathcal{D}}(r_{i})} T_4 \ll \Bigg(\sum_{r|q}\frac{1}{r} \sum_{r_1, \ldots, r_k}\bigg(\prod \frac{s^{\omega(r_i)}r_i}{\phi_{\mathcal{D}}(r_{i})} \bigg)^2\Bigg)^{1/2} \Bigg(\sum_{r|q}\frac{1}{r} \sum_{r_1, \ldots, r_k} T^{2}_4\Bigg)^{1/2}.
\end{equation}
The first factor on the right is made larger as it runs over all $k$-tuples for which $r_i|r$. The larger expression is
\begin{equation}
\label{5.12}
\sum_{r|q}\frac{1}{r} \prod_{p|r}\bigg(1+\big(\frac{sp}{\phi_{\mathcal{D}}(p)}\big)^2\bigg)^k= \prod_{p|q}\Bigg(1+ \frac{\bigg(1+\big(\frac{sp}{\phi_{\mathcal{D}}(p)}\big)^2\bigg)^k}{p}\Bigg) \ll P^{-(s^2 +1)^k}.
\end{equation}
The second factor has been treated precisely in \cite[pp. 324-325]{mv} and it is smaller than $h^{-2/7k}$. By combining \eqref{5.9}, \eqref{5.10}, \eqref{5.11}  and \eqref{5.12} we complete the proof of \eqref{eq9}. Note that we have just sketched the key ideas of the proof, the interested reader can find further details in \cite{mv}. To finish the proof of the Theorem \ref{thm4} we appeal to Lemma 3.1 in \cite{F.A}.  Let $q_1= \prod_{\substack{p|q \\p\leq y }} p$ and $q_2= \prod_{\substack{p|q \\ p>y}} p$, where $y\geq h^k$.
 We set $P_i= \frac{\phi(q_i)}{q_i}$ for $i=1, 2$. Then \cite[Lemma 3.1]{F.A} states that
\begin{align*}
M^{\mathcal{D}}_k(q, h) \ll q(hP^s)^{[k/2]}+qh({P)^s} + qh^{k/2}P_1^{-2^{ks}+ks}P_2^{sk}.
\end{align*}
This lemma is obtained by combining two different estimates of $M^{\mathcal{D}}_k(q, h)$: an exponential estimate and a probabilistic estimate. The exponential estimate stated in \cite[Lemma 1.2]{F.A} gives
$$ M^{\mathcal{D}}_k(q, h) \ll  qh^{k/2} P^{-2^{ks}+ks}.$$
Here we use the estimate \eqref{eq9}, instead of the above estimate and we derive:
$$M^{\mathcal{D}}_k(q, h) \ll_k q(hP^s)^{[k/2]}+qh{P^s} + qh^{k/2}P_{1}^{sk-s^2\frac{k}{2}}\big(1+ h^{-\frac{1}{7k}}P_1^{-(s+1)^k}\big)P_2^{sk}.$$
Now by considering $y=h^{k}$, we have \eqref{eq8} and for $\displaystyle{h<e^{\frac{1}{k P^{1/s}}}}$, we have \eqref{eq7},
which completes the proof.
\end{proof}
\subsection*{Acknowledgement}

For providing the proof of Equation \eqref{Fourier} I would like to thank Will Sawin and Noam Elkies.
For helpful comments and feedback, I am grateful to
Amir Akbary,
Adam Harper, Tristan Freiberg, Kevin Henriot
and
the anonymous referee.
For guidance and helpful discussions, I am grateful to my supervisor Nathan Ng.

{\it E-mail address}: farzad.aryan@uleth.ca
\end{document}